\documentclass[a4paper,12pt]{amsart}
\usepackage{amsmath,amssymb,colordvi}
\usepackage{graphicx}
\usepackage{setspace}

\newtheorem{theorem}{Theorem}

\newtheorem{example}[theorem]{\it Example}
\newtheorem{proposition}[theorem]{Proposition}
\newtheorem{definition}[theorem]{Definition}

\font\tenBb=msbm10 \font\sevenBb=msbm7 \font\fiveBb=msbm5
\newfam\Bbfam \textfont\Bbfam=\tenBb \scriptfont\Bbfam=\sevenBb
\scriptscriptfont\Bbfam=\fiveBb

\def\Bb{\fam\Bbfam\tenBb}
\def\R{{\Bb R}}
\def\C{{\Bb C}}

\usepackage{tikz}
\usepackage{xparse}
\NewDocumentCommand\UpArrow{O{2.0ex} O{black}}{%
	\mathrel{\tikz[baseline] \draw [->, line width=0.5pt, #2] (0,0) -- ++(0,#1);}
}

\NewDocumentCommand\DownArrow{O{2.0ex} O{black}}{%
	\mathrel{\tikz[baseline] \draw [<-, line width=0.5pt, #2] (0,0) -- ++(0,#1);}
}

\begin{document}
\title{Copulas and preserver problems }

 \author{A. Sani}
  \address{Universit\'e Ibn Zohr, Facult\'e des Sciences, D\'epartement de Math\'ematiques, Agadir, Maroc}
  \email{ahmedsani82@gmail.com}

\author{L. Karbil}
\address{Universit\'e Hassan I, ENSEM, Casablanca, Maroc}
\email{l.karbil@gmail.com }


\footnote{\today}



\begin{abstract}
	Preserver problems   concern the characterization of  operators on
	general spaces that leave invariant  some categories of subsets or ratios. The most known in the mathematical literature are those of linear  preserver problems (LPP) which date back to the ninth century. Here, we treat the preserver problem in the recent and emerging field of copulas. Precisely, we prove that \emph{copula property} is preserved uniquely under increasing transformations.

\end{abstract}

\maketitle

\section{Introduction}
\setcounter{theorem}{0}
 \setcounter{equation}{0}

\vspace{2cm}

\begin{spacing}{1.5}
The study of Preserver problems and linear preserver problems (LPPs for short) is a vast field and an interesting topic for mathematical researches.  For a historical overview on preserver problems and techniques used in their treatment, the important article of\emph{ Chi-Kwong Li} and \emph{Nam-Kiu Tsingt} \cite{Li92} is recommended and advisable. Recently the linear case  has been intensively  studied. We start by making precise this notion of linear preservation. By 'linear' we mean here  linear or at least  additive
maps on Banach algebras that leave invariant  certain subsets. It is a kind of invariance under linear operations. The mappings defined on Banach spaces which preserve some properties such as invertibility are widely studied mainly in \cite{JS86}, \cite{Ka70} and   \cite{{Jo92}}  and  in another context by \cite{Mb07} where surjective linear maps preserving the set
of Fredholm operators are discussed. In the same spirit, M. Mbkhta et.al \cite{Mb14} pursued in a strong work the description of all additive and  surjective continuous
maps in the algebra of all bounded linear operators acting on a complex
separable infinite-dimensional Hilbert space, preserving some  interesting classes of operators acting such as the operators of finite
ascent, the operators of finite descent or  Drazin invertible ones.  As a pedagogical illustration, one may consider in the elementary algebra the preservation of determinant function on $\mathcal M_n(\C)$ by the linear mappings $\Phi_{P,Q}: A\rightarrow PAQ$ where $P$ and $Q$ are two fixed square matrices. The spectacular result in this direction was the converse characterization by Frobenus \cite{Fr97}. \\
 Recently, the attention is paid to non linear preserver problems since they do not restrict the study to linear operators and allow the preservation  treatment of some emergent stochastic and statistical properties. For example, in \cite{Cu11}, some unnecessary linear continuous bijections on the set of all self adjoint operators  maps $Bs(H)$ which preserve the star order are characterized.\\
In this paper we are concerned with preserver problem of \emph{copula} property. For the best of our knowledge, this problem is less investigated in the literature. The  sense of preservation as well as necessary reminders and complements  on copulas will be made precise bellow. The introductory note  on copulas is based on  the   unsurmountable book of R. Nelson \cite{Ne06}. \\
The term  \emph{copula} appeared for first time in the edifying paper of Abe Sklar at 1959 \cite{Sk59} as an answer to   Maurice Fr\'echet's probability problem. It   became henceforth an efficient tool when statisticians deal with ex-changeability, dependence or symmetry and asymmetry  problems (see  \cite{Sani19} for a functional point of view of these concepts). It is well known that if $C(.,.)$ is a copula then for every nondecreasing functions $\phi$ and $\psi$ on $\R$ the function $C(\phi(.),\psi(.))$ defines a new copula. To describe this situation, we say that \emph{copula property} is preserved by nondecreasing functions. The present paper clarifies and discusses whether only this type of mappings have this property or not.\\
After this brief preamble, the next section is devoted to recall important notions on copulas and  gives a general reminder of most important results on preserving problems.\\
The paper is organized as follows: After the introduction, some reminders on mathematical concepts used here are reviewed as preliminaries. The second section reshuffles essential results on increasing functions in a general framework before specifying the case of $2-$functions on which we limit our work. In the third section,  some properties which are preserved by increasing transformations are discussed. Classical results on this topic will be overflown. At the last section,  some examples are given as illustration.

\section{Preliminaries}

We start by giving an overview on copulas. One may restrict the study on the bivariate case without affecting the generality of the problem. Let $I$ denote from now on the interval $[0,1]$. We assume in addition that all copulas treated here are supported on  $I^2$ which means in probabilistic terms that the laws are diffuse.
\begin{definition}\label{Def1}
	A copula $C$ is a function on $I^2$ into $I=[0,1]$ which satisfies the following conditions for all $(u,v)\in I^2$:
	\begin{enumerate}
		\item $C(0,v)=C(u,0)=0.$
		\item $C(1,v)=v$ and $C(u,1)=u.$
		\item the 2-increasing property: $C(u_2,v_2)-C(u_2,v_1)-C(u_1,v_2)+C(u_1,v_1)\ge0.$
	\end{enumerate}
	
\end{definition}
Let $X$ and $Y$ be two continuous random variables with distribution
functions $F_X$ and $F_Y$. Let $H$ be their  joint distribution function. Then according to Sklar theorem(see \cite{Ne06}), there exists a unique copula $C_{XY}$ such that $H(x,y)=C_{XY}(F_X(x),F_Y(y)).$ The continuity hypothesis on $X$ and $Y$ is assumed just to guarantee  uniqueness of $C_{XY}.$ Copulas are popular in high-dimensional statistical applications as they allow us to  model easily  the distribution of random vectors by estimating marginals and copula separately. There are many \textit{parametric} copula families in the literature, which usually have parameters that control the strength of dependence or symmetry.\\
 
\noindent \textbf{\underline{Question:}} What are transformations $\phi$ and $\psi$  of random variables $X$ and $Y$ which preserve properties (1)-(3) in Definition \ref{Def1} for the  new mapping $C_{\phi(X)\psi(Y)}$?\\
A historical answer to this important question of preservation was given by R. Nelsen in \cite{Ne06}. Precisely, the following theorem was proved therein

\begin{theorem}\label{Nelson_preservation}
Let $X$ and $Y$ be continuous random variables with copula
$C_{XY}$ . If $\phi$ and $\psi$ are strictly increasing on their respective ranks $Ran(X)$ and  $Rank(Y)$
then $$C_{\phi(X)\psi(Y)}=C_{XY}$$
\end{theorem}
For convenience, we will rewrite  (see section \ref{section4} below) the original proof of this important theorem which gives a first bridge between linear and non linear preservation problems. 
\noindent Nelson comments on this result  saying that \emph{copula property is preserved.} \\
The discussion about correlation coefficient given so far in the current paper (see section 3 below) gives an elementary example for a quick familiarization of our concern.\\
To clarify how to preserve the copula property, a deep understanding of $2-$increasing behavior of real  functions defined on a given area of $\R^2$ is necessary. In the following section we give enough tools to describe  $2$-increasing functions and understand why conditions (1) and (2) in definition (\ref{Def1}) are necessary to serve as a  bridge between monotonicity of one and two variables.


\section{On 2-increasing functions}

The interest of $2-$monotone functions does not need to be clarified. For our purposes, it suffices to refer to any document tracing the concordance problems in statistics. Mathematically speaking, the suitable framework to a general treatment is the lattice Banach spaces by which one understands  ordered Banach spaces with a positive cone
$P\subset X.$ A lot of topics in classical fields of mathematics used the notion of monotonicity to prove interesting results mainly on fixed points. One may see for example \cite{Hu88} where  the existence of maximal and minimal fixed points are proved. Anyhow for application to copulas, the natural framework should be functions defined on $\R^2.$ It seems obvious that for a given function $f:\ (x,y)\mapsto f(x,y)$ the simultaneous monotony of $f$ with respect to $(x,y)$ will be equivalent to ordinary monotony of partial function $x\mapsto f(x,y)$ and $y\mapsto f(x,y)$ with respect to $x$ and $y$ respectively. Far from being so, counterexamples prove a negative conclusion. Let us begin by clarifying this unpredictable statement

\begin{definition}
	A function $f$ of two variables defined on $\Omega \subset \R^2$ is increasing in volume or $2-$increasing, if for all $(x_2,y_2)\ge (x_1,y_1)$ we have
	
	\begin{equation}\label{volume_monotony}
	[f(x_2,y_2)-f(x_1,y_2)]-[f(x_2,y_1)-f(x_1,y_1)]\ge 0.
	\end{equation}
\end{definition}

As mentionned above, not any 2-increasing function in volume is  nondecreasing in each argument. The examples are not lacking. It suffices to consider the following classical  one:
\begin{example}
let $f$ be the function defined on $I^2$ by: $f(a,b)=(2a-1)(2b-1).$ Then $f$ is $2-$increasing but not  nondecreasing of the first component $a$ for arbitrary values of $b\in I$ nor of the second one $b$ for arbitrary values of $a\in I.$ On the other hand the function $g: (a,b)\rightarrow \max(a,b)$ is nondecreasing on each argument $a$ and $b$ but is far to be $2-$increasing since $ g(1,1)-g(0,1))-(g(1,0)-g(0,0))=-1.$
\end{example}
Even if these negative results occur, it is easy to verify that partial variations of every $2-$increasing function are nondecreasing one-valued mappings. More precisely, we have
\begin{proposition}\label{increasing_property}
	Let $f$ be a $2-$increasing function. Then for all $(b_1,b_2)\in I^2$ such that $b_1<b_2$, the partial function $a\mapsto f(a,b_2)-f(a,b_1)$ is a non decreasing function. Analogously, for every $(a_1,a_2)\in I^2$ such that $a_1<a_2,$ the function $b\mapsto f(a_2,b)-f(a_1,b)$ is nondecreasing on $I.$
	 
\end{proposition}

The proof is a direct consequence of \emph{volume monotony} (\ref{volume_monotony}). At this point, one may see easily that condition (1) in definition (\ref{Def1}) is fundamental to ensure that partial functions $u\mapsto C(u,v)$ and $v\mapsto C(u,v)$ are nondecreasing ones. In probabilistic words, these conditions explain that the margins $u\mapsto C(u,1)$ and $v\mapsto C(1,v)$ of the copula $C$ (which are uniform distributions) are diffuse and supported entirely by the interval $I.$ Nelson \cite{Ne06} took advantage of this quality to prove an interesting analytic property of copulas seen as functions of two variables. He proved that they are Lipschitz continuous as the following proposition states

\begin{proposition}
	Let $C$ be a copula as given by definition (\ref{Def1}). For all $ (x_1, y_2)$, $ (x_1, y_2)$ in $I^2$ we have

\begin{equation}\label{lipsh_cont}
|C(x_2, y_2) - C(x_1, y_1) | \le |x_2-x_1| +|y_2-y_1|
\end{equation}

\end{proposition}

The fact that a copula enjoys  this property is of great interest. One may see \cite{Sa18} for some clarifications and applications to asymmetry problem of copulas. The reason that we recall this result lies in its importance in computing ratios (coefficients) like $\rho$ of Kendall or $\tau$ of Spearman given by formulas $$\tau=\frac{2}{n(n-1)}\sum_{i<j} sgn(x_j-x_i) sgn(y_j-y_i) \  \mbox{and} \  \rho={1-{\frac {6\sum d_{i}^{2}}{n(n^{2}-1)}}},$$
where $sgn$ is the sign function and $d_i$ is the difference between two ranks of each observation $(x_i,y_i)$ of statistical collected  data $(X,Y)$.\\ It is known that 2-increasing functions preserve such ratios, this yields the interest of such functions.

\section{Improvement of Nelson classical result}\label{section4}
The historical and simple Nelson's reproach  to  linear regression ratios was the instability with usual non linear transformations like square and/or exponential  functions. This kind of mappings perturbs widely  regression dependence given by the correlation coefficient $\rho=\frac{Cov(X,Y)}{\sigma(X)\sigma(Y)}.$ This latter coefficient is just preserved under linear transformations. Indeed, it is elementary to check that for any two  linear scalar transformations $\phi$ and $\psi$(i.e $\phi(x)=ax$ and $\psi(y)=by$ for certain $(a,b)\in R^2$), one has

\begin{eqnarray}\label{rho_preserved1}
	\rho{(\phi(X),\psi(Y))}&=& \frac{Cov(\phi (X),\psi(Y))}{\sigma(\phi(X))\sigma(\psi(Y))}\\
	&=& \frac{Cov(aX,bY)}{ab\sigma(X)\sigma(Y)})\\
	&=& \frac{Cov(X,Y)}{\sigma(X)\sigma(Y)}\\
\label{rho_preserved2}	&= & \rho(X,Y).
\end{eqnarray}

The natural question in this case and in our framework is whether this kind of transformation is the only one which ensures the preservation of the linear regression dependence. An immediate answer to the fact of failure of preservation under non linear transformations follows from the simple consideration of $\phi(t)=\psi(t)=t^2.$ So  for a standard normal random variable $X$, one has $\rho_{(\phi,\psi)}(X,X)=3\ne 1=\rho(X,Y).$\\
Assume now that the coefficient $\rho$ is invariant under all convex transformations $\phi$ and $\psi.$ Without loss of generality and thanks to usual rescaling procedure, one may suppose that the random variables $X$ and $Y$ are centred and reduced. In such case, the assumed preserving property will be expressed namely as:
\begin{equation}\label{linear_preserv}
\frac{Cov(\phi (X),\psi(Y))}{\sigma(\phi(X))\sigma(\psi(Y))}=Cov(X,Y)\ \text{for all convex functions}\ \phi \ \text{and}\ \psi .
\end{equation}
 
\textbf{\underline{Question}}: Under the hypothesis (\ref{linear_preserv}), what will be the general form of functions $\phi$ and $\psi$? More precisely, if we assume that (\ref{linear_preserv}) holds for every choice of random variables $X$ and $Y$, what should be the expressions of $\phi(x)$ and $\psi(y)$ when $x$ and $y$ vary on $\R$?\\
The elementary calculus in (\ref{rho_preserved1})-(\ref{rho_preserved2})  says in other words that all  linear functions belong to the desired class. Are they alone? \\

\noindent The following proposition gives a positive answer
\begin{proposition}\label{prop_lin_pres_rho}
	Let $\phi$ and $\psi$  real valued functions such that (\ref{linear_preserv}) holds. Then 
	$$
	\forall (x,y)\in \R^2: \  \  \phi(x)=ax+b \  \text{and} \  \psi(y)=a'y+b'
	$$
	 for some real constants $a$, $a'$, $b$ and $b'.$ 
\end{proposition}

\begin{proof}
All random variables are considered on the  same probability space 
$(\Omega; \mathcal F; \mathbb{P})$. This is possible thanks to Skohorod theorem (see \cite{Ber15} among others) stated for other concerns that we do not develop here (mainly types of stochastic convergence). After reducing and  centering the random variables $X$ and $Y$ the condition (\ref{linear_preserv}) is equivalent to 
	$$
	E[(\phi(X)-E(\phi(X))). (\psi(Y)-E(\psi(Y)))]=\sigma(\phi(X))\sigma(\psi(Y))E(XY).
	$$
	For the sake of simplicity and readability,  let us denote $\lambda=\sigma(\phi(X))\sigma(\psi(Y)).$ In order to dodge uniqueness of antecedent we assume in addition that $\psi$ is invertible. For a particular choice $\psi^{-1}(Y)$ instead of $Y$, and taking into account that $E(Y)=0$ this yields 
		$$
	E[(\phi(X)-E(\phi(X))). Y]=\lambda E(XY).
	$$
	Since this last equality occurs for all random variables $Y$, classical results on Lebesgue integration theory lead to  
	$$(\phi(X)-E(\phi(X)))=\lambda X,\ \mathbb{P}-a.e.$$
	
	On the set $\{ Y\ne0\}=\{ \omega, Y(\omega)\ne 0\}$ that we may assume with null $\mathbb{P}-$measure, one has $\phi(X)-E(\phi(X))=\lambda X.$ The claim follows by taking $a=\lambda$ and $b=E(\phi(X))$ as long as this latter quantity is well defined. \\
	The manifest symmetry in (\ref{linear_preserv}) ensures that an analogous argument prevails for the function $\psi$.
	\end{proof}

We hope now to give a similar argument for the \emph{copula preserver problem}. To do so, let's go back to copulas and recall the classical result of Nelson \cite[pages 25 and 26]{Ne06} concerning the preservation of copula property under strictly increasing transformations
\begin{proposition}
 Let $X$ and $Y$ be continuous random variables associated with the copula
$C_{(X,Y)}$ . If $\phi$ and $\psi$ are respectively strictly increasing on $Ran(X)$ and $Ran(Y)$,
then $C_{(X,Y)}=C_{(\phi(X),\psi(Y)}$ . 	
	
\end{proposition}	
 A proof may be found in \cite{Ne06}. As promised and for convenience we rewrite briefly the argument given therein. 
 \begin{proof}
 	In the sequel if $Z$ is a random variable, we denote from now on $F_Z$ its distribution function.\\
 	As $X$ and $Y$ are assumed continuous, the range problem  is excluded. On the other hand, for all $(a,b)$ in $\R^2$ and all strictly increasing functions $\phi$ and $\psi$, the random variables $\phi(X)$ and $\psi(Y)$ are continuous and the direct computation leads to

\begin{eqnarray*}
	C_{(\phi(X),\psi(Y)}(F_{\phi(X)}(a),F_{\psi(Y)}(b))&=& P((\phi(X)\le a,\psi(Y)\le b)\\
&=& P(X\le \phi^{-1}(a),Y\le \psi^{-1}(b))\\
&=& C_{(X,Y)}(F_X o \phi^{-1}(a),F_Yo \psi^{-1}(b))\\
&= & C_{(X,Y)}(F_{\phi(X)}(a),F_{\psi(Y)}(b)).
\end{eqnarray*} 

The continuity assumption ensures that every couple $(u,v)\in [0,1]^2$ may be written $(u,v)=(F_{\phi(X)}(a),F_{\psi(Y)}(b))$ for some convenient $(a,b)\in \R^2.$ This yields
$$ \forall (u,v)\in [0,1]^2: \  C_{(\phi(X),\psi(Y)}(u,v)=C_{(X,Y)}(u,v).$$	
 \end{proof}	
 A real  challenge consists in proving the converse in the following sense: If the copula property is preserved by the couple $(\phi;\psi)$ for all continuous random variables $X$ and $Y$, the functions  $\phi$ and $\psi$ are automatically increasing? if not what should be their general form?\\
 
 A first and simple analysis leads to the following remark:  not every real value linear functions  $\phi$ and $\psi$ warrant the preservation of copula property. This is obvious since if one or both of $\phi$ and $\psi$ are not increasing, the preservation property fails according to \cite[Theorem 2.4.4]{Ne06} (choose $a$ or $a'$ strictly negative in proposition \ref{prop_lin_pres_rho}). Hence the problem is far from being  LPP one. \\
 The following proposition gives a partial answer that we hope to improve in the future.
  
 \begin{proposition}
 	Let $\phi$ and $\psi$ be two real value  functions such that for a given random vector $(X,Y)$, the identity  $C_{(X,Y)}=C_{(\phi(X),\psi(Y))}$  holds.
 	Then 	
 	$\phi$ and $\psi$ are increasing.
 \end{proposition}
 
 \begin{proof}
 Consider two real functions $\phi$ and $\psi$ such that for all random variables $X$ and $Y$ we have $C_{(X,Y)}(u,v)=C_{(\phi(X),\psi(Y))}(u,v)$ for all couple $(u,v)\in [0,1]^2.$\\
 Assume first that $\phi$ is continuously differentiable and $\psi(0)=0$. According to Definition (\ref{Def1}) and Proposition (\ref{increasing_property}), one may write  
 

\begin{eqnarray*}\label{}
\frac{\partial C_{(\phi(X),\psi(Y))}(u,v)}{\partial u}&=&\frac{\partial C_{(\phi(X),\psi(Y))}(u,v)-C_{(\phi(X),\psi(Y))}(u,0)}{\partial u}\\
\frac{\partial C_{(\phi(X),\psi(Y))}(u,v)}{\partial u}&=& \frac{ \partial C_{(X,Y)}((\phi(u),\psi(v)))-C_{(X,Y)}((\phi(u),\psi(0)))}{\partial x}\\
&=& \phi'(u)\frac{\partial C_{(X,Y)}((\phi(u),\psi(v)))}{\partial x}
\end{eqnarray*} 

The assumption on $C_{(X,Y)}$ and $C_{(\phi(X),\psi(Y))}$ ensures that for all $u\in [0,1]$, one has $\phi'(u)\ge 0.$ In equivalent words, the mapping $\phi$ is nondecreasing. For symmetry reasons in the third item of (\ref{Def1}) an analogous argument is valid to prove that $\psi$ is also an increasing function. \\
For   function $\phi$ without nonzero value at  the origin, one may consider $\Psi=\psi-\psi(0)$ which shares the same monotony with $\psi$ and verifies $\Psi(0)=0.$\\
For general statement ($\phi$ is not necessarily differentiable), one may argue  as follows:
\begin{enumerate}
	\item Almost every real number is Lebesgue point of $\phi$.
\item  For such point $t$ we have $\phi'(t)\ge 0.$
\end{enumerate}

To complete the claim, a topological argument and density of regular functions on $[0,1]$ allows the conclusion.
 	
 \end{proof}


\end{spacing}
\end{document}